\documentclass[12pt]{amsart}
\usepackage{verbatim}
\usepackage{amssymb}
\usepackage{upref}
\usepackage[all]{xy}
\usepackage{color}
\usepackage{tikz,pgf}

\allowdisplaybreaks

\newcommand{\Hs}{Hilbert system}
\newcommand{\Hm}{Hilbert matrix}
\newcommand{\Hhyphenm}{Hilbert-matrix}
\newcommand{\Hms}{Hilbert matrices}
\newcommand{\te}{tensor equivalent}
\newcommand{\tenoun}{tensor equivalence}

\newtheorem{thm}{Theorem}[section]
\newtheorem*{thm*}{Theorem}

\newtheorem{cor}[thm]{Corollary}
\newtheorem{prop}[thm]{Proposition}
\newtheorem{obs}[thm]{Observation}

\theoremstyle{definition}
\newtheorem{defn}[thm]{Definition}
\newtheorem{q}[thm]{Question}
\newtheorem{qs}[thm]{Questions}

\newtheorem{ex}[thm]{Example}

\newtheorem*{notn*}{Notation}

\newtheorem*{hyp*}{Hypothesis}

\newtheorem{rem}[thm]{Remark}

\numberwithin{equation}{section}

\newcommand{\secref}[1]{Section~\textup{\ref{#1}}}

\newcommand{\thmref}[1]{Theorem~\textup{\ref{#1}}}
\newcommand{\corref}[1]{Corollary~\textup{\ref{#1}}}

\newcommand{\obsref}[1]{Observation~\textup{\ref{#1}}}
\newcommand{\defnref}[1]{Definition~\textup{\ref{#1}}}

\newcommand{\qref}[1]{Question~\textup{\ref{#1}}}

\newcommand{\exref}[1]{Example~\textup{\ref{#1}}}

\newcommand{\midtext}[1]{\quad\text{#1}\quad}
\newcommand{\righttext}[1]{\quad\text{#1 }}
\renewcommand{\and}{\midtext{and}}
\renewcommand{\for}{\righttext{for}}

\newcommand{\N}{\mathbb N}

\newcommand{\C}{\mathbb C}
\newcommand{\T}{\mathbb T}

\newcommand{\KK}{\mathcal K}
\newcommand{\HH}{\mathcal H}

\newcommand{\LL}{\mathcal L}
\newcommand{\FF}{\mathcal F}
\newcommand{\EE}{\mathcal E}
\newcommand{\OO}{\mathcal O}

\newcommand{\Chi}{\raisebox{2pt}{\ensuremath{\chi}}}
\renewcommand{\epsilon}{\varepsilon}

\DeclareMathOperator{\aut}{Aut}

\DeclareMathOperator{\ad}{Ad}

\DeclareMathOperator{\supp}{supp}

\DeclareMathOperator{\pic}{Pic}

\DeclareMathOperator*{\spn}{span}
\DeclareMathOperator*{\clspn}{\overline{\spn}}

\newcommand{\id}{\text{\textup{id}}}

\newcommand{\<}{\langle}
\renewcommand{\>}{\rangle}

\newcommand{\inv}{^{-1}}

\newcommand{\iso}{\overset{\cong}{\longrightarrow}}
\renewcommand{\bar}{\overline}

\newcommand{\wilde}{\widetilde}

\newcommand{\csp}{cor\-re\-spond\-ence} 
\newcommand{\dml}{dim\-en\-sion\-al}

\begin{document}

\title[Obstructions]{Obstructions to a general characterization of graph correspondences}
\author[Kaliszewski, Patani, and Quigg]
{S.~Kaliszewski, Nura Patani, and John Quigg}
\address [S.~Kaliszewski] {School of Mathematical and Statistical Sciences, Arizona State University, Tempe, Arizona 85287} \email{kaliszewski@asu.edu}\
\address[Nura Patani]{School of Mathematical and Statistical Sciences, Arizona State University, Tempe, Arizona 85287} \email{nura.patani@asu.edu}
\address[John Quigg]{School of Mathematical and Statistical Sciences, Arizona State University, Tempe, Arizona 85287} \email{quigg@asu.edu}

\date{January 11, 2012}

\begin{abstract}
For a countable discrete space $V$,
every nondegenerate separable $C^*$-\csp\ over $c_0(V)$ 
is isomorphic to one coming from a directed graph with vertex set $V$.  
In this paper we demonstrate why the analogous characterizations 
fail to hold for higher-rank graphs 
(where one considers product systems of $C^*$-\csp s)
and for topological graphs 
(where $V$ is locally compact Hausdorff), 
and we discuss the obstructions that arise.
\end{abstract}


\maketitle

\section{Introduction}

Given a directed graph $E$ with countable vertex set $V$,
one can construct the \emph{graph correspondence} $X_E$,
which is a $C^*$-\csp\ over $c_0(V)$
that encodes the structure of $E$ in a particularly nice way: 
indeed, we have recently shown
in \cite{kpq1} that the assignment $E\mapsto X_E$ gives rise to an
essentially surjective functor from a category of directed graphs
with vertex set~$V$ to a category of $c_0(V)$-\csp s.
In particular, every separable, nondegenerate $c_0(V)$-\csp\ 
is isomorphic to a graph correspondence.

In this paper we investigate this issue 
--- characterizing the $C^*$-\csp s that come from graphs ---
in two more-general contexts: that of \emph{$k$-graphs}, 
which are higher-\dml\  analogs of directed graphs; 
and that of \emph{topological graphs}, which are continuous analogs. 
It turns out that in each case the most na\"ive characterization fails:
there are $C^*$-\csp s that 
are not isomorphic to any graph correspondence.
However, we can in special cases identify obstructions
that may lead to a general characterization.

We refer to \cite{raeburngraph} for the basic definitions and notations
regarding directed graphs, and their associated $C^*$-correspondences, $k$-graphs, and topological $k$-graphs.

In \secref{sec-one}, we describe product systems associated to $k$-graphs.  In particular, we discuss how a $k$-graph may be viewed as a product system over $\N^k$ of directed graphs and define the associated \emph{$k$-graph correspondence}, which is a product system of $C^*$-correspondences.  We discuss how each of these product systems may be realized in terms of its \emph{skeleton}  and show how the factorization of paths in the $k$-graph is encoded in the skeleton of the $k$-graph correspondence.

In \secref{sec-Hsystem}, we define \emph{Hilbert matrices}
and \emph{Hilbert systems} and show that Hilbert systems are essentially the same as product systems of $C^*$-correspondences.  We restrict our attention to the case $k=2$ in \secref{sec-2gr} and establish a characterization of the product systems over $\N^2$ of $c_0(V)$-correspondences arising from 2-graphs.  Further, we specialize to the setting $V=\{v\}$ and identify an isomorphism invariant for product systems over $\N^2$ of one-dimensional Hilbert spaces that allows us to identify precisely when such a product system is isomorphic to a 2-graph correspondence.

In \secref{sec-topgr}, we turn our attention to topological graphs with vertex space $V$,
and initially focus on Hilbert $A$-modules, where $A=C_0(V)$.  To each Hilbert $A$-module $X$ is associated a Hilbert bundle $p:\HH\to V$.  We show that the existence of a continuous choice of orthonormal bases for this Hilbert bundle is equivalent to $X$ being isomorphic to the Hilbert $A$-module constructed from a topological graph using the source map to implement the right $A$-module structure.  To incorporate the left $A$-module structure, we consider the representations $\pi_v$ coming from the left module action of $A$ on each fibre of the associated Hilbert bundle.  For an $A$-correspondence to arise from a topological graph, the continuous choice of orthonormal bases must diagonalize the representation $\pi_v$ for each $v\in V$.

To more closely study this, we restrict our attention to the setting where the associated Hilbert bundle has 1-dimensional fibres.  For such an $A$-correspondence, being isomorphic to a topological graph correspondence is equivalent to the associated Hilbert bundle being trivial.  Further specializing to the setting of $A$-imprimitivity bimodules, we identify the \emph{Picard invariant} of an $A$-imprimitivity bimodule~$X$ as the obstruction to $X$ being isomorphic to a topological graph correspondence.

\section{$k$-graphs and product systems over $\N^k$}
\label{sec-one}

Throughout this section, we consider a fixed countable set $V$.
After a brief discussion of product systems,
we show how each $k$-graph with vertex set $V$
gives rise to a product system over $\N^k$
of $c_0(V)$-\csp s, which we call a \emph{$k$-graph correspondence}.
We show that in general, any product system of $c_0(V)$-\csp s
is determined, up to isomorphism, by a set of basic
data we call its \emph{skeleton},
and we describe the skeleton of a $k$-graph correspondence
in terms of the structure of the $k$-graph.  

If $E$ and $F$ are directed graphs with vertex set $V$,
the \emph{fibred product} $E*F$ 
is the graph with vertex set $V$,
edge set
\[
 (E*F)^1 = \{ (e,f) \in E^1\times F^1 \mid s_E(e) = r_F(f) \},
\]
and range and source maps $r$ and $s$ given by
\[
 r(e,f) = r_E(e)
\quad\text{and}\quad
s(e,f) = s_F(f).
\]
If $\phi\colon E\to E'$ and $\psi\colon F\to F'$ are 
graph homomorphisms that agree on~$V$,
then the fibred product homomorphism $\phi*\psi\colon E*F\to E'*F'$
defined on edges by $\phi*\psi(e,f)=(\phi(e),\psi(f))$
and on vertices by $v\mapsto \phi(v)=\psi(v)$
is easily seen to be a graph homomorphism.
In particular, if $E'$ and $F'$ are directed graphs with vertex set $V$ and if $\phi,\psi$ are \emph{vertex-fixing isomorphisms} (i.e., $\phi(v)=\psi(v)=v$), then $\phi*\psi$ is a graph isomorphism. 
(This construction makes the set 
of directed graphs having vertex set $V$
into a \emph{tensor groupoid} in the language of \cite{fowlersims_artin}.)

If $S$ is a countable semigroup with identity $e$,
then a \emph{product system over $S$ of graphs on $V$}
is a collection $E = \{ E_s \mid s\in S \}$ of 
directed graphs $E_s$ with vertex set $V$,
together with a collection $\alpha = \{ \alpha_{s,t} \mid s,t\in S \}$
of vertex-fixing graph isomorphisms
$\alpha_{s,t}\colon E_s*E_t \to E_{st}$
such that
$E_e = (V,V,\id_V,\id_V)$;
$\alpha_{e,s}\colon E_e*E_s \to E_s$
and $\alpha_{s,e}\colon E_s*E_e \to E_s$
are the natural maps given on edges by
$(r(e),e) \mapsto e$ and $(e,s(e)) \mapsto e$ for each $s\in S$;
and the \emph{associativity condition}
\begin{equation}\label{alpha}
 \alpha_{rs,t}\circ (\alpha_{r,s}*\id_t) = \alpha_{r,st}\circ(\id_r*\alpha_{s,t})
\end{equation}
holds {for all $r,s,t\in S$,}
where $\id_t$ is the identity map on $E_t$.
(This definition is a special case of \cite[Definition~1.1]{fowlersims_artin}.) 

Now let $A$ be a $C^*$-algebra.
A \emph{product system over $S$ of $A$-\csp s}
is a collection $X = \{ X_s \mid s\in S \}$
of $A$-\csp s,
together with a collection $\beta = \{ \beta_{s,t} \mid s,t\in S \}$
of $A$-\csp\ isomorphisms 
$\beta_{s,t}\colon X_s\otimes_A X_t\to X_{st}$
such that
$X_e = A$ (viewed in the natural way as an $A$-\csp);
$\beta_{e,s}\colon X_e\otimes_A X_s\to X_s$
and $\beta_{s,e}\colon X_s\otimes_A X_e\to X_s$
are the natural maps determined by
$a\otimes\xi\mapsto a\cdot\xi$
and $\xi\otimes a\mapsto \xi\cdot a$
for each $s\in S$;
and the associativity condition
$\beta_{rs,t}\circ(\beta_{r,s}\otimes\id_t) = \beta_{r,st}\circ(\id_r\otimes\beta_{s,t})$
holds for all $r,s,t\in S$.
(Again, see \cite[Definition~1.1]{fowlersims_artin}.)

As observed in \cite{fowlersims_artin}, 
product systems over $\N^k$ of graphs
are essentially the same as $k$-graphs. 
In more detail,
suppose $(\Lambda, d)$ is a \emph{$k$-graph with vertex set $V$}:
so $\Lambda$ is a countable category,
$d\colon \Lambda\to \N^k$ is the {degree functor},
and we have identified the object set $d\inv(0)$
with~$V$.
For each $m\in \N^k$, 
the set $d\inv(m)$ is the edge set of a directed graph $E_m$ with vertex set $V$ and range and source maps inherited from $\Lambda$.  
The collection $E = \{ E_m \mid m\in \N^k \}$
together with the vertex-fixing
isomorphisms $\alpha_{m,n}\colon E_m*E_n\to E_{m+n}$ given on edges by
\[
 \alpha_{m,n}(\mu,\nu) = \mu\nu
\]
is then a product system over $\N^k$ of graphs on $V$;
moreover, every such product system arises in this way
from a $k$-graph with vertex set $V$.

Retaining the above notation, for each $m\in\N^k$,
let $X_m$ be the $c_0(V)$-\csp\ associated to the graph $E_m$;
recall (from \cite{raeburngraph}, for example)
that by definition
\[
 X_m = \bigl\{ \xi\colon E_m^1\to \C \mid 
\text{the map $v\mapsto \sum_{s(e)=v} |\xi(e)|^2$ is in $c_0(V)$} \bigr\},
\]
with module actions and $c_0(V)$-valued inner product given by
\[
 (f\cdot\xi\cdot g)(e) = f(r(e))\xi(e)g(s(e))
\quad\text{and}\quad
\langle \xi,\eta\rangle(v) = \sum_{s(e)=v} \overline{\xi(e)}\eta(e).
\]
Note that $X_m$ is densely spanned by the set $\{ \Chi_e \mid e\in E_m^1 \}$,
where $\Chi_e$ denotes the characteristic function of $\{ e \}$.
Further, if $p_v\in c_0(V)$ denotes the characteristic function of a vertex $v\in V$, 
then for any $e\in E_m^1$ and $f\in E_n^1$ we have
\[
 \Chi_e\otimes \Chi_f = \Chi_e \cdot p_{s(e)} \otimes p_{r(f)}\cdot\Chi_f
= \Chi_e\cdot p_{s(e)}p_{r(f)}\otimes \Chi_f = 0
\]
in $X_m\otimes_{c_0(V)} X_n$ unless $s(e)=r(f)$; 
thus the balanced tensor product is densely spanned by the set
\[
 \{ \Chi_e \otimes \Chi_f \mid (e,f)\in (E_m*E_n)^1 \}.
\]
For each $m,n\in \N^k$ let $\beta_{m,n}\colon X_m\otimes_{c_0(V)}X_n\to X_{m+n}$
be the isomorphism determined by
\begin{equation}\label{beta}
 \beta_{m,n}( \Chi_e\otimes \Chi_f ) = \Chi_{\alpha_{m,n}(e,f)}
= \Chi_{ef}
\quad\text{ for $(e,f)\in (E_m*E_n)^1$.}
\end{equation}
Then $X = \{ X_m \mid m\in \N^k \}$ with $\beta = \{ \beta_{m,n}\mid m,n\in \N^k\}$
is a product system over $\N^k$ of $c_0(V)$-\csp s
(see \cite[Examples~1.5]{fowlersims_artin}).
We call $(X,\beta)$ the \emph{$k$-graph correspondence} 
associated to  $(\Lambda,d)$.

Suppose $(X,\beta)$ is any product system over $\N^k$ of $A$-\csp s,
and let $\{ e_i \mid 1\leq i\leq k \}$ denote the standard basis for $\N^k$.
Then setting 
\[
Y_i = X_{e_i}
\quad\text{and}\quad
T_{i,j} = \beta_{e_j,e_i}\inv\circ\beta_{e_i,e_j}\colon 
Y_i\otimes_A Y_j \to Y_j\otimes_A Y_i
\]
gives a collection $Y = \{ Y_i \mid 1\leq i\leq k \}$ of $A$-\csp s 
and a collection $T = \{ T_{i,j} \mid 1\leq i<j\leq k \}$ of $A$-\csp\ isomorphisms
such that the \emph{hexagonal equation}
\begin{equation}\label{hex} 
\begin{split}
(T_{j,\ell}\otimes\id_i)(\id_j \otimes T_{i,\ell})&(T_{i,j}\otimes\id_\ell) \\
&= (\id_\ell\otimes T_{i,j})(T_{i,\ell}\otimes\id_j)(\id_i\otimes T_{j,\ell})
\end{split}
\end{equation}
holds for all $1\leq i<j<\ell\leq k$,
where $\id_i$ is the identity map on $Y_i$. 
(Note that both sides of~\eqref{hex} are isomorphisms 
$Y_i\otimes_A Y_j\otimes_A Y_\ell \to Y_\ell\otimes_A Y_j\otimes_A Y_i$.)
We call the pair $(Y,T)$ the \emph{skeleton} of $(X,\beta)$.

It follows from  \cite[Proposition~2.11]{fowlersims_artin} 
(see also Remark~\ref{ijk} below) that 
a product system is uniquely determined, up to isomorphism, by its 
skeleton.  More precisely, if $(Z,\gamma)$ is another product system
over $\N^k$ of $A$-\csp s,
with skeleton $(W,R)$, then $(X,\beta)\cong (Z,\gamma)$ 
if and only if $(Y,T)\cong (W,R)$ in the sense 
that there exist isomorphisms $\theta_i\colon Y_i\to W_i$
such that the following diagram commutes for each $1\leq i<j\leq k$:
 \begin{center}
 \begin{tikzpicture}
 \node (a) at (-2,2) {$Y_i\otimes Y_j$};
 \node (b) at (2,2) {$Y_j\otimes Y_i$};
 \node (c) at (-2,0) {$W_i\otimes W_j$};
 \node (d) at (2,0) {$W_j\otimes W_i$.};
 \draw[->] (a) -- node[above] {$T_{i,j}$} (b);
 \draw[->] (c) -- node[above] {$R_{i,j}$} (d);
 \draw[->] (a) -- node[left] {$\theta_i\otimes\theta_j$} (c);
 \draw[->] (b) -- node[right] {$\theta_j\otimes\theta_i$} (d);
 \end{tikzpicture}
\end{center}

\begin{rem}\label{ijk}
Observe that for $k\leq 2$, there are no hexagonal equations~\eqref{hex}, 
and for $k=1$ there are no $T_{i,j}$'s.
Also, defining $T_{j,i}=T_{i,j}\inv$ for $i<j$, we have a collection $\{T_{i,j}\mid 1\le i\ne j\le k\}$ of correspondence isomorphisms that satisfy the hexagonal equations for all distinct $i,j,l$,
as in \cite{fowlersims_artin}.
It is convenient to note that we only need the $T_{i,j}$'s for $i<j$.
\end{rem}

The following proposition shows how the skeleton encodes the ``commuting squares''
--- that is, the factorizations of paths of degree $e_i+e_j$ ---
of a $k$-graph.

\begin{prop}\label{skeleton}
Suppose $(\Lambda, d)$ is a $k$-graph with vertex set $V$,
let $(E,\alpha)$ be the associated product system over $\N^k$ of graphs,
let $(X,\beta)$ be the associated $k$-graph correspondence, 
and let $(Y,T)$ be the skeleton of $(X,\beta)$.
Then for each $1\leq i<j\leq k$
and each $e\in E_{e_i}^1$ and $f\in E_{e_j}^1$ we have
\[
 T_{i,j}( \Chi_e\otimes \Chi_f ) = \Chi_{\tilde f}\otimes \Chi_{\tilde e},
\]
where $\tilde f$ and $\tilde e$ are the unique edges
in $E_{e_j}$ and $E_{e_i}$, respectively,
such that $\tilde f \tilde e = ef$ in~$E_{e_i+e_j}$. 
\end{prop}

\begin{proof}
By~\eqref{beta},
we have
$\beta_{e_i,e_j}(\Chi_e\otimes \Chi_f) = \Chi_{ef}$ in $X_{e_i+e_j}$,
and similarly
$\beta_{e_j,e_i}(\Chi_{\tilde f}\otimes\Chi_{\tilde e}) 
= \Chi_{\tilde f\tilde e}
= \Chi_{ef}$.
Since $T_{i,j} = \beta_{e_j,e_i}\inv\circ\beta_{e_i,e_j}$ by definition,
the result follows.
\end{proof}

Since we are interested in determining which product systems come from $k$-graphs, and since the correspondences in the product system coming from a $k$-graph are automatically nondegenerate, we will restrict attention to product systems of nondegenerate correspondences.

\section{Product systems over $\N^2$}

To analyze product systems of correspondences over $c_0(V)$ in more detail, we need to decompose using the minimal projections $p_v:=\Chi_{\{v\}}$ for $v\in V$. For arbitrary $k$ this would involve an enormous amount of book-keeping, and would be further complicated by the hexagonal equation~\eqref{hex}. 
Therefore we now specialize to the case $k=2$. Some of what we will do (particularly concerning product systems of 1-dimensional correspondences) can be carried over to the general case.

So, the skeleton $(Y,T)$ of our product system $(X,\beta)$ over $\N^2$ has $Y=(Y_1,Y_2)$; since we will only be working with two correspondences, we will simplify the notation by letting
\[
(Y_1,Y_2)=(Y,Z),
\]
Also, we only have one isomorphism $T_{1,2}$, so we will just write this as $T$.
Thus, $T:Y\otimes_A Z\iso Z\otimes_A Y$ is an isomorphism of $A$-correspondences, and
we denote the skeleton of the product system $(X,\beta)$ by
$(Y,Z,T)$.

We use similar notational conventions for isomorphisms between product systems over $\N^2$: if $(X',\beta')$ is another product system, with skeleton $(Y',Z',T')$, an isomorphism from the first skeleton to the second is an ordered pair $(\theta,\tau)$ of $A$-correspondence isomorphisms
\[
\theta:Y\iso Y'\midtext{and}\tau:Z\iso Z'
\]
such that the diagram
\[
\xymatrix{
Y\otimes_A Z \ar[r]^-T \ar[d]_{\theta\otimes\tau}
&Z\otimes_A Y \ar[d]^{\tau\otimes\theta}
\\
Y'\otimes_A Z' \ar[r]_-{T'}
&Z'\otimes_A Y'
}
\]
commutes.

\section{\Hs s}
\label{sec-Hsystem}

We continue to consider a fixed countable set~$V$.

\begin{defn}
A \emph{\Hm\ over $V$} is a family $H=\{H_{uv}\}_{u,v\in V}$ of Hilbert spaces.
\end{defn}

Unless otherwise specified, all our \Hms\ will be over $V$.

\begin{defn}
If $H$ and $K$ are \Hms, an \emph{isomorphism} of $H$ onto $K$ is a family $S=\{S_{uv}\}_{u,v\in V}$, where each $S_{uv}$ is a unitary operator from $H_{uv}$ onto $K_{uv}$.
\end{defn}

\begin{defn}
The \emph{product} of two \Hms\ $H$ and $K$ is the \Hm\ $H*K$, where
\[
(H*K)_{uv}=\bigoplus_{x\in V}(H_{ux}\otimes K_{xv})
\]
\end{defn}

\begin{defn}
A \emph{\Hs\ over $V$} is a triple $(H,K,S)$, where $H$ and $K$ is are \Hms\ and $S:H*K\to K*H$ is an isomorphism.
\end{defn}

Since our \Hms\ will be over $V$ by default, so will our \Hs s.

\begin{rem}
The existence of an isomorphism $H*K\cong K*H$ for two \Hms\ $H$ and $K$ 
is equivalent to the following relation among the dimensions of the component Hilbert spaces:
\[
\sum_{x\in V}(\dim H_{ux})(\dim K_{xv})=\sum_{x\in V}(\dim K_{ux})(\dim H_{xv})\righttext{for all}u,v\in V.
\]
\end{rem}

\begin{defn}
An \emph{isomorphism} $(\sigma,\tau):(H,K,S)\to (H',K',S')$ between \Hs s comprises isomorphisms $\sigma:H\to H'$ and $\tau:K\to K'$ making the diagram
\[
\xymatrix{
H*K \ar[r]^-{S_{uv}} \ar[d]_{\sigma*\tau}
&K*h \ar[d]^{\tau*\sigma}
\\
H'*K' \ar[r]_-{S'_{uv}}
&K'*H'
}
\]
commute, where the fibered product $\sigma*\tau$ has component unitaries
\[
(\sigma*\tau)_{uv}=\bigoplus_{x\in V}(\sigma_{ux}\otimes\tau_{xv}),
\]
and similarly for $\tau*\sigma$.
\end{defn}

We will establish an equivalence between product systems and \Hs s. Unsurprisingly, we will work with the skeletons of the product systems. Recall that we put $A=c_0(V)$.
First we define the \Hs\ associated to a product system.

\begin{defn}\label{Y to H}
The \emph{\Hm\ associated to} an $A$-correspondence $Y$ is given by $H=\{H_{uv}\}_{u,v\in V}$ where
\[
H_{uv}=p_u\cdot Y\cdot p_v.
\]
Here $p_u=\Chi_{\{u\}}$ is the characteristic function of the singleton set $\{u\}$, regarded as a minimal projection in $A$, and where the 1-dimensional ideal $\C p_v$ is identified with the complex numbers $\C$.
\end{defn}

\begin{ex}
If $Y=X(E)$ is the graph correspondence associated to a directed graph $E$ with vertex set $V$, then
\[
H_{uv}=\ell^2(uE^1v).
\]
\end{ex}

\begin{obs}\label{T to S}
Let $X$ be a product system over $\N^2$ of $A$-correspondences, with skeleton $(Y,Z,T)$, let $H$ and $K$ be the \Hms\ associated to $Y$ and $Z$, respectively, as in \defnref{Y to H}.
It follows from the constructions that $H*K$ is the \Hm\ associated to the $A$-correspondence $Y\otimes_A Z$, and similarly for $K*H$ and $Z\otimes_A Y$.
Moreover, for each $u,v\in V$ the restriction $T|_{(H*K)_{uv}}$ gives a unitary map
\[
S_{uv}:(H*K)_{uv}\to (K*H)_{uv}.
\]
\end{obs}

\begin{ex}\label{graph system}
In \obsref{T to S}, if $X$ is the graph associated to a 2-graph, with associated product system $\{E_n\}_{n\in\N^2}$ of graphs, then
\[
(H*K)_{uv}=\clspn\{\Chi_e\otimes\Chi_f:(e,f)\in  u(E_{e_1}*E_{e_2})^1v\}.
\]
Moreover,
\[
S_{uv}(\Chi_e\otimes\Chi_f)=\Chi_{\wilde f}\otimes\Chi_{\wilde e},
\]
where $\tilde f$ and $\tilde e$ are the unique edges
in $E_{e_j}$ and $E_{e_i}$, respectively,
such that $\tilde f \tilde e = ef$ in~$E_{e_i+e_j}$.
\end{ex}

\begin{defn}\label{corres to Hs}
Let $(Y,Z,T)$ be the skeleton of a product system $X$ over $\N^2$ of $A$-correspondences. The \emph{\Hs\ associated to $X$} is $(H,K,S)$,  where $S=\{S_{uv}\}$ is the family of unitaries defined in \obsref{T to S}.
\end{defn}

We now turn to the task of defining the product system associated to a \Hs.

\begin{obs}\label{H to corres}
Let $H$ be a \Hm\ over $V$. Then for every $u,v\in V$ we can define an $A$-correspondence structure on $H_{uv}$ as follows: for $\xi,\eta\in H_{uv}$ and $f\in A$ we let
\begin{itemize}
\item $\xi\cdot f=\xi f(v)$;

\item $f\cdot \xi=f(u)\xi$;

\item $\<\xi,\eta\>_A=\<\xi,\eta\> p_v$,
\end{itemize}
where on the right-hand side of the last equation the unadorned inner product $\<\xi,\eta\>$ is the one given by the Hilbert space $H_{uv}$.\footnote{Note that in order to facilitate the connection with correspondences we adopt the physicists' convention that the Hilbert-space inner product is linear in the second variable!}
\end{obs}

\begin{defn}\label{H to Y}
Let $H$ be a \Hm\ over $V$, and let the Hilbert spaces $H_{uv}$ be regarded as $A$-correspondences as in \obsref{H to corres}.
Then the \emph{$A$-correspondence associated to $H$} is
\[
Y:=\bigoplus_{u,v\in V}H_{uv},
\]
regarded as a direct sum of $H$-correspondences.
\end{defn}

\begin{obs}\label{iso H to Y}
Let $H$ and $K$ be \Hms\ over $V$, and let $Y$ and $Z$ be the $A$-correspondences associated to $H$ and $K$, respectively, as in \defnref{H to Y}. Further let $S:H\to K$ be a \Hhyphenm\ isomorphism. Then each $S_{uv}:H_{uv}\to K_{uv}$ can be regarded as an $A$-correspondence isomorphism, and the direct sum
\[
T=\bigoplus_{u,v\in V}S_{uv}:Y\to Z
\]
is also an $A$-correspondence isomorphism.
\end{obs}

\begin{defn}\label{associated iso}
With the notation of \obsref{iso H to Y}, we call $T:Y\to Z$ the \emph{$A$-correspondence isomorphism associated to $S$}.
\end{defn}

\begin{defn}\label{Hs to corres}
Let $(H,K,S)$ be a \Hs, and let $Y$ and $Z$ be the $A$-correspondences associated to $H$ and $K$, respectively, as in \defnref{H to Y}.
Then the \emph{product system associated to $(H,K,S)$}
is the unique product system determined by the skeleton $(Y,Z,T)$, where $T:Y\otimes_A Z\to Z\otimes_A Y$ is the $A$-correspondence isomorphism associated to the \Hhyphenm\ isomorphism $S:H*K\to K*H$ as in \defnref{associated iso}.
\end{defn}

\begin{thm}\label{equivalence}
\begin{enumerate}
\item Let $X$ be a product system over $\N^2$ of $A$-correspondences, let $(H,K,S)$ be the associated \Hs, and let $X'$ be the product system associated to $(H,K,S)$.
Then $X\cong X'$.

\item Let $(H,K,S)$ be a \Hs, let $X$ be the associated product system, and let $(H',K',S')$ be the \Hs\ associated to $X$. Then $(H,K,S)\cong (H',K',S')$.
\end{enumerate}
\end{thm}

\begin{proof}
(i) If $(Y,Z,T)$ and $(Y',Z',T')$ are the corresponding skeletons of $X$ and $X'$, respectively, then $Y$ is the internal direct sum of the $A$-correspondences $H_{uv}$, while $Y'$ is their external direct sum, and similarly for $Z$, $Z'$, and the $K_{uv}$'s. The canonical isomorphism between internal and external direct sums is easily seen to give an isomorphism $X\cong X'$.

(ii) If $(Y,Z,T)$ is the skeleton of $X$, then $Y$ is the external direct sum of the $H_{uv}$'s while each $H_{uv}$ is a sub-correspondence of $Y'$, and similarly for $Z$, $K_{uv}$, and $K'_{uv}$. 
It is easy to check that if we take, for every $u,v\in V$, $\sigma_{uv}$ to be
the canonical inclusion of $H_{uv}$ into $Y$, and similarly for $\tau_{uv}$ including $K_{uv}$ into $Z$, then $(\sigma,\tau):(H,K,S)\to (H',K',S')$ is an isomorphism.
\end{proof}

\begin{rem}
\thmref{equivalence} says that the passage between product systems and \Hs s are inverse processes up to isomorphism. In fact, this could be promoted to an equivalence of categories, but since we have no applications in mind we will not make this precise.
\end{rem}

\section{The obstruction for $2$-graphs}
\label{sec-2gr}

For the case $k=2$, our characterization problem becomes the following:

\begin{q}
Which product systems are isomorphic to 2-graph correspondences?
\end{q}

The fact that product systems over $\N^2$ of $A$-correspondences are completely characterized by their associated \Hs s will lead us to one answer to the above question.

\begin{defn}
Let $H$ and $K$ be \Hms.
For $u,v\in V$ let $\EE_{u,v}$ and $\FF_{u,v}$ be orthonormal bases of the Hilbert spaces $H_{u,v}$ and $K_{u,v}$, respectively.
Put
\[
(\EE*\FF)_{u,v}=\{e\otimes f: e\in\EE_{u,x},f\in\FF_{x,v} \text{ for some } x\in V\}.
\]

\end{defn}

\begin{obs}
With the above notation, $(\EE*\FF)_{u,v}$ is an orthonormal basis of $(H*K)_{u,v}$.
\end{obs}

\begin{thm}\label{2-graph decomp}
Let $X$ be a product system over $\N^2$ of $A$-correspondences, and let $(H,K,S)$ be the associated \Hs.
Then $X\cong X(\Lambda)$ for some 2-graph $\Lambda$ with vertex set $V$ if and only if there are choices $\EE_{u,v}$ and $\FF_{u,v}$ of orthonormal bases of the $H_{u,v}$ and $K_{u,v}$, respectively, such that for all $u,v\in V$ we have
\[
S_{u,v}\bigl((\EE*\FF)_{u,v}\bigr)=(\FF*\EE)_{u,v}.
\]
\end{thm}

\begin{proof}
It is routine to verify that the stated property is an isomorphism invariant of \Hs s, and \exref{graph system} shows that it is satisfied when $X$ is a 2-graph correspondence.
\end{proof}

\begin{rem}
Relaxing the criterion in \thmref{2-graph decomp}
a little bit, we can make contact with the $k$-graph cohomology of \cite{KPSHomology}, although we only pursue it in the case $k=2$. Recall that on a 2-graph $\Lambda$ 
an element of $Z^2(\Lambda,\T)$ 
is just a function 
$\phi:\Lambda^{(1,1)}\to\T$, 
and $\phi,\phi'\in H^2(\Lambda,\T)$ are \emph{cohomologous} if there exists 
$\alpha\in Z^1(\Lambda,\T)$, i.e., 
a function $\alpha:\Lambda^{(1,0)}\cup \Lambda^{(0,1)}\to\T$ 
satisfying $\alpha(e)\alpha(f)=\alpha(f')\alpha(e')$ whenever $ef$ and $f'e'$ are the two edge-path factorizations of an element of $\Lambda^{(1,1)}$, such that 
\begin{equation}\label{cohom}
\alpha(e)\alpha(f)\phi(ef)=\alpha(f')\alpha(e')\phi'(f'e')
\end{equation}
in all such cases.

Consider the following relaxation of 
the criterion in \thmref{2-graph decomp}:
suppose that there are choices $\EE_{u,v}$ and $\FF_{u,v}$ of orthonormal bases of $H_{u,v}$ and $K_{u,v}$, respectively, such that 
for all $u,v\in V$ and $(e,f)\in( \EE*\FF)_{uv}$ 
there exists
$(f',e')\in (\FF*\EE)_{uv}$ such that
$S_{uv}(e\otimes f)$ is a scalar multiple of $f'\otimes e'$.

Equivalently, suppose that we have a bijection $\theta:\EE*\FF\to\FF*\EE$ and a map
$\phi_0:\EE*\FF\to\T$
such that if $(f',e')=\theta(e,f)$ then 
\begin{equation}\label{twisted S}
S_{uv}(e\otimes f)=\phi_0(e,f)(f'\otimes e').
\end{equation}
By \cite[Discussion in Section~6]{kp:kgraph} there is a unique 2-graph $\Lambda$,
with component directed graphs $\Lambda^{(1,0)}=\EE$ and $\Lambda^{(0,1)}=\FF$,
such that for all $\lambda\in\Lambda^{(1,1)}$, if $(e,f)$ is the unique element of $\EE*\FF$ such that $\lambda=ef$, then, with the above notation, $\lambda=f'e'$ is the other edge-path factorization.
It then follows that there is a unique $\phi\in Z^2(\Lambda,\T)$ such that for all $\lambda\in \Lambda^{(1,1)}$, if $\lambda=ef$ for $(e,f)\in\EE*\FF$ then
\[
\phi(\lambda)=\phi_0(e,f).
\]

We can recover the \Hs\ $(Y,Z,S)$, and hence the product system $X$, up to isomorphism from $\Lambda$ and $\phi$: for example we can take $Y$ to be the completion of the free vector space generated by $\EE$, completed with respect to the unique inner product in which $\EE$ is orthonormal, and similarly for $Z$ and $\FF$,
and then the cocycle $\phi$ clearly determines the unitaries $S_{uv}$.

The Cuntz-Pimsner algebra $\OO_X$ is isomorphic to the twisted 2-graph algebra $C^*_\phi(\Lambda)$ of \cite{KPSHomology}.
\cite[Proposition~7.6]{KPSHomology} shows that if $\phi$ is homologous to $\phi'$ then $C^*_\phi(\Lambda)\cong C^*_{\phi'}(\Lambda)$; in fact, a little more is true: the associated product systems are isomorphic. To see this, suppose $\phi$ and $\phi'$ are related as in \eqref{cohom}, 
and let $X$ and $X'$ be the associated product systems, with associated \Hs s $(H,K,S)$ and$(H,K,S')$, respectively ---
to be clear, we build the \Hms\ $H$ and $K$ such that, 
for every $u,v\in V$, 
$u(\Lambda^{(1,0)})v$ and $u(\Lambda^{(0,1)})v$ 
are orthonormal bases of $H_{uv}$ and $K_{uv}$, respectively; 
the only difference between $X$ and $X'$ is that the isomorphism 
$S:H*K\to K*H$ is built using $\phi$ and $S':H*K\to K*H$ is built using $\phi'$.
For each $u,v\in V$ let $\sigma_{uv}:H_{uv}\to H'_{uv}$ be the unique unitary such that
\[
\sigma_{uv}e=\alpha(e)e\for e\in \EE,
\]
and similarly for $\tau_{uv}:K_{uv}\to K'_{uv}$. It is routine to check that $(\sigma,\tau):(H,K,S)\to (H,K,S')$ is an isomorphism.
\end{rem}

\subsection*{Product systems of Hilbert spaces}

\thmref{2-graph decomp} does not seem to give a very practical test. The main difficulty in applying it would lie in analyzing the unitaries
\[
S_{uv}:\bigoplus_{x\in V}(H_{ux}\otimes K_{xv})\to \bigoplus_{x\in V}(K_{ux}\otimes H_{xv});
\]
$S_{uv}$ need not have any relation to the direct-sum decompositions.

In order to obtain a more practical test, we specialize to the case of one vertex, i.e., $V=\{v\}$. 
Then we 
can regard the \Hs\ $(H,K,S)$ as comprising a single unitary $S$ between Hilbert spaces $H$ and $K$,
and $X$ becomes a product system of Hilbert spaces.
Moreover, if $X'$ is another such product system of Hilbert spaces, with associated \Hs\ $(H',K',S')$, then $X\cong X'$ if and only if there are unitaries $\sigma:H\to H'$ and $\tau:K\to K'$ making the diagram
\[
\xymatrix{
H\otimes K \ar[r]^-S \ar[d]_{\sigma\otimes\tau}
&K\otimes H \ar[d]^{\tau\otimes\sigma}
\\
H'\otimes K' \ar[r]_{S'}
&K'\otimes H'
}
\]
commute.

We will get a more manageable 
invariant by composing with the flip: let $\Sigma:H\otimes K\to K\otimes H$ be the flip map, and define
\[
\omega=\Sigma S,
\]
which is a unitary operator on the Hilbert space $H\otimes K$.

\begin{defn}
With the above notation, we call $\omega$ the \emph{flipped unitary} associated to the product system $X$ over $\N^2$ of Hilbert spaces.
\end{defn}

\begin{defn}\label{equivalent}
Let $H$ and $K$ be Hilbert spaces. 
Say that two unitary operators $\omega$ and $\omega'$ on $H\otimes K$ are \emph{\te} if there exist unitaries $\sigma\in \LL(H)$ and $\tau\in \LL(K)$ such that
\[
\omega'=\ad(\sigma\otimes\tau)\omega.
\]
\end{defn}

The class of the
flipped unitary 
under \tenoun\
gives a complete isomorphism invariant for product systems over $\N^2$ of Hilbert spaces.
Using this invariant, the criterion of \thmref{2-graph decomp} becomes the following

\begin{cor}\label{2-graph omega}
A product system $X$ over $\N^2$ of Hilbert spaces is isomorphic to $X(\Lambda)$ for some 2-graph $\Lambda$ with a single vertex if and only if its flipped unitary is \te\ in the sense of \defnref{equivalent} to an elementary tensor $U_H\otimes U_K$ of permutation unitaries on $H$ and $K$.
\end{cor}

Here we say a unitary operator on a Hilbert space is a \emph{permutation unitary} if it can be represented as a permutation matrix with respect to some orthonormal basis.

The above results take the following particularly simple form for a single vertex and one-dimensional Hilbert spaces.

\begin{cor}\label{circle}
Let $X$ be a product system over $\N^2$ of one-dimensional Hilbert spaces, with flipped unitary $\omega$. Then $\omega$ may be identified with an element of the unit circle $\T$, and the isomorphism classes of such product systems are parameterized by $\T$.

Moreover, $X$ is isomorphic to a 2-graph correspondence if and only if $\omega=1$.
\end{cor}

\begin{proof}
Let $(H,K,S)$ be the \Hs\ associated to $X$. Since $H$ and $K$ are one-dimensional, so is $H\otimes K$, so the unitary operator $\omega$ is a scalar multiple of the identity map on a one-dimensional Hilbert space, and this scalar must be an element of $\T$. In fact, $\omega$ can be identified with this element of $\T$, and since any unitary operator on $H\otimes K$ can occur as a flipped unitary of a suitable $X$, any element of $\T$ can arise in this way. It is obvious that  tensor equivalence on the set of unitary operators on the one-dimensional Hilbert space $H\otimes K$ is just equality. For the last statement just observe that an elementary tensor of permutation unitaries on the one-dimensional Hilbert spaces $H$ and $K$ is associated to the element 1 of $\T$.
\end{proof}

\begin{rem}
Thus, even in the trivial case of one-dimensional correspondences over $\C$, there are uncountably many product systems over $\N^2$ that are not isomorphic to 2-graph correspondences.

It is interesting to identify the $C^*$-algebra $\OO_X$ of a product system over $\N^2$ of one-dimensional Hilbert spaces $H$ and $K$: let $\omega$ be the flipped unitary, which we identify with an element of $\T$ as in \corref{circle}. Choosing unit vectors $e$ and $f$ in $H$ and $K$, respectively, it is easy to see using induction that Cuntz-Pimsner covariant representations of the product system $X$ correspond to an assignment $e\mapsto u$ and $f\mapsto v$ of unitaries satisfying the relation $uv=\omega vu$, and so, letting $\omega=e^{i\theta}$, it follows that $\OO_X$ is isomorphic to the rotation algebra with angle $\theta$ \cite[Proposition~5.1.5]{nura}.
\end{rem}

\begin{rem}
It follows from \cite{fowlersims_artin} that the skeleton of a product system $(X,\beta)$ over $\N$
of $c_0(V)$-correspondences is just a single correspondence~$Y$;
then by \cite[Theorem~1.1]{kpq1} we have $Y\cong X_E$ for some directed graph $E$, 
and thus we quickly conclude that $(X,\beta)$ is isomorphic to a $1$-graph correspondence.
More generally, every
product system $(X,\beta)$ over a finitely-generated free semigroup $S$
of $c_0(V)$-correspondences 
is isomorphic to one associated with a product system over $S$ of graphs.
To see this, note that, in the terminology of Fowler-Sims \cite{fowlersims_artin}, 
$S$ is a right-angled Artin semigroup with no commutation relations, 
so the skeleton $Y$ of $(X,\beta)$ has no $T_{i,j}$'s; 
again by \cite{kpq1} each $Y_i$ is isomorphic to a graph correspondence $X_{E_i}$, 
and the $E_i$'s form the skeleton of a product system over $S$ of graphs, 
whose associated product system of correspondences can be checked to be isomorphic to $(X,\beta)$ 
by examining the skeletons.
Thus when the semigroup $S$ is free and finitely-generated 
there is no obstruction to a product system over $S$ of $c_0(V)$-correspondences 
coming from a product system over $S$ of graphs.
\end{rem}

\section{Topological graphs}
\label{sec-topgr}

Now let $V$ be a fixed locally compact Hausdorff space, and let $A=C_0(V)$.
For any topological graph $E=(E^1,V,r,s)$ with vertex space $V$,
let $X(E)$ denote the associated $A$-correspondence.
In this section our characterization problem becomes the following.

\begin{q}
Which $A$-correspondences are isomorphic to $X(E)$ for some topological graph with vertex space $V$?
\end{q}

This question can be fruitfully separated into two parts:

\begin{qs}\label{questions}
(1)
Which Hilbert $A$-modules are isomorphic to $X(E^1,s)$ for some local homeomorphism $s:E^1\to V$?
Here we use self-explanatory notation: if $E$ is a topological graph, then forgetting the range map $r:E^1\to V$, and forgetting the left $A$-module structure, the $A$-correspondence $X(E)$ becomes a Hilbert $A$-module $X(E^1,s)$ that only depends upon the source map $s$, and can be formed from any locally compact Hausdorff space $E^1$ and any local homeomorphism $s:E^1\to V$.

(2)
For Hilbert $A$-modules of the form $X=X(E^1,s)$, which left $A$-module structures make $X$ into a topological graph correspondence $X(E)$ for a suitable range map $r$?
\end{qs}

We need to use the correspondence between Hilbert $A$-modules and continuous Hilbert bundles (for the case of compact $V$, see, e.g., \cite{takahashi1, takahashi2}, also the references \cite{DixmierDouady, dg} are useful), which we briefly recall here in simplified form. 
First, if $p:\HH\to V$ is a continuous Hilbert bundle, then the space $\Gamma_0(\HH)$ of $C_0$-sections becomes a Hilbert $A$-module with operations
\begin{align*}
\<\xi,\eta\>_A(v)&=\<\xi(v),\eta(v)\>\\
(\xi\cdot f)(v)&=\xi(v)f(v)
\end{align*}
for $\xi,\eta\in\Gamma_0(\HH)$ and $f\in A$.
In the other direction,
let $X$ be a Hilbert $A$-module. For each $v\in V$, define a positive semidefinite hermitian form $\<\cdot,\cdot\>_v$ on $X$ by
\[
\<\xi,\eta\>_v=\<\xi,\eta\>_A(v).
\]
Then the Hausdorff completion of $X$ with respect to $\<\cdot,\cdot\>_v$ is a Hilbert space $H_v$.
Let $Q_v:X\to H_v$ be the associated map.
Let $\HH=\bigsqcup_{v\in V}H_v$, and define
$\Phi:X\to \prod_{v\in V}H_v$
by
\[
(\Phi\xi)(v)=Q_v\xi.
\]
Then there is a unique topology on $\HH$ making it into a continuous Hilbert bundle over $V$
such that $\Phi$ is an isometric isomorphism of $X$ onto $\Gamma_0(\HH)$,
and then in fact $\Phi$ is a Hilbert $A$-module isomorphism.
Letting $J_v=\{f\in A:f(v)=0\}$ be the associated maximal ideal, we can equivalently identify
\[
H_v=X/(X\cdot J_v),
\]
which, by the general theory of Hilbert modules, is naturally a Hilbert $A/J_v$-module, and then $Q_v:X\to H_v$ is just the quotient map.

Now consider the case $X=X(E^1,s)$.
Here we can take
\[
H_v=\ell^2(s\inv(v)),
\]
and then $Q_v:X\to H_v$ is given by $Q_v\xi=\xi|_{s\inv(v)}$,
and
$\Phi$ is given by
\[
(\Phi\xi)(v)=\sum_{s(e)=v}\xi(e)\Chi_e,
\]
where we write $\Chi_e$ for the characteristic function of the singleton~$\{e\}$.

We will give one answer to \qref{questions} (1) using the following concept:

\begin{defn}\label{ONB}
Let $p:\HH\to V$ be a continuous Hilbert bundle. We say that a subset $\EE\subset\HH$ is a \emph{continuous choice of orthonormal bases} if
\begin{enumerate}
\item for each $v\in V$, the intersection $\EE_v:=\EE\cap H_v$ is an orthonormal basis for the Hilbert space $H_v$;

\item the restriction of the bundle projection $p$ to $\EE$ is a local homeomorphism.
\end{enumerate}
\end{defn}

\begin{thm}\label{hilbert thm}
A Hilbert $A$-module $X$ is isomorphic to $X(E^1,s)$ for some local homeomorphism $s:E^1\to V$ if and only if the associated Hilbert bundle $p:\HH\to V$ has a continuous choice of orthonormal bases.
\end{thm}

\begin{proof}
First assume that $X=X(E^1,s)$ for a local homeomorphism $s:E^1\to V$.
Define $\gamma:E^1\to \HH$ by
\[
\gamma(e)=\Chi_e,
\]
and let $\EE=\gamma(E^1)$.

\textbf{Claim:}
$\gamma$ is a homeomorphism of $E^1$ onto $\EE$.
To see this, note first of all that
$\gamma$ is certainly injective. To see that $\gamma$ is continuous, let $e_0\in E^1$, and choose a neighborhood $U$ of $e_0$ on which $s$ restricts to a homeomorphism. Then choose $\xi\in C_c(E^1)$ such that $\supp\xi\subset U$ and $\xi$ is identically $1$ on a smaller neighborhood $W$ of $e$. Then
\[
(\Phi\xi)(s(e))=\xi(e)\Chi_e\midtext{for all}e\in U,
\]
so for all $e\in W$ we have
\[
\gamma(e)=\Chi_e=(\Phi\xi)(s(e)),
\]
which is continuous in $e$ because $\Phi\xi$ is a continuous section of $\HH$.

To see that $\gamma$ is open onto its image, let $N$ be an open neighborhood of $e_0$ in $E^1$, and if necessary shrink $N$ so that $s$ maps $N$ homeomorphically onto an open set in $V$.
Put $W=\gamma(N)$.
Then on $W$ the inverse of $\gamma$ is given by
\[
\gamma\inv=\tau\circ p,
\]
where $\tau:s(N)\to N$ is the inverse of the homeomorphism
\[
s|_N:N\to s(N)
\]
and $p:\HH\to V$ is the bundle projection.
Since both $p$ and $\tau$ are continuous, we see that
\[
W=(\gamma\inv)\inv(N)=p\inv(\tau\inv(N))\cap \gamma(E^1)
\]
is open, and the Claim is verified.

For each $v\in V$, clearly
\[
\EE_v:=\EE\cap H_v=\{\Chi_e\}_{e\in s\inv(v)}
\]
is an orthonormal bases of the fibre $H_v=\ell^1(s\inv(v))$ of $\HH$.
The commutative diagram
\[
\xymatrix{
E^1 \ar[r]^\gamma \ar[dr]_s
&\EE \ar[d]^{p|_\EE}
\\
&V
}
\]
shows that $p|_\EE$ is a local homeomorphism because $\gamma$ is a homeomorphism.

Conversely, suppose that we have a continuous choice $\EE$ of orthonormal bases for $\HH$.
Let $s=p|_\EE$, which is a local homeomorphism by hypothesis.
We will construct an isomorphism $\Psi:X(\EE,s)\to X$ of Hilbert $A$-modules.
For $\xi\in X(\EE,s)$ define a section $\Psi\xi$ of the Hilbert bundle $\HH$ by
\[
(\Psi\xi)(v)=\sum_{e\in\EE_v}\xi(e)e,
\]
where as above $\EE_v=\EE\cap H_v$ is an orthonormal basis of the fibre $H_v$. Note that $\EE_v=s\inv(v)$.
We have
\begin{align*}
\|(\Psi\xi)(v)\|^2
&=\sum_{e\in s\inv(v)}|\xi(e)|^2
=\<\xi,\xi\>_{C_0(V)}(v),
\end{align*}
so $\Psi\xi\in \Gamma_0(\HH)$.
Clearly $\Psi:X(\EE,s)\to \Gamma_0(\HH)$ is linear. We show that it preserves inner products: for $\xi,\eta\in X(\EE,s)$ and $v\in V$ we have
\begin{align*}
\<\Psi\xi,\Psi\eta\>_{C_0(V)}(v)
&=\bigl\<(\Psi\xi)(v),(\Psi\eta)(v)\bigr\>
\\&=\sum_{e\in s\inv(v)}\bigl\<\xi(e)e,\eta(e)e\bigr\>
\\&=\sum_{e\in s\inv(v)}\bar{\xi(e)}\eta(e)
\\&=\<\xi,\eta\>_{C_0(V)}(v).
\end{align*}
It follows that $\Psi$ preserves the right $A$-module structures.
It now remains to show that $\Psi$ is surjective,
and for this it suffices to show that the set $\Psi(X(\EE,s))$ of sections is a $C_0(V)$-submodule that is total in each fibre: the first holds since $\Psi$ preserves the right $A$-module structures, and for the second we see that for each $v\in V$ we have
\[
\{(\Psi\xi)(v):\xi\in X(\EE,s)\}\midtext{is total in}H_v,
\]
because it contains $\EE_v$, since we can take $\xi\in X(\EE,s)$ that is $1$ at any given $e\in \EE_v$ and is $0$ at all other $e'\in s\inv(v)$.
\end{proof}

Now we incorporate the left $A$-module structure.
Given an $A$-correspondence $X$, we know from the above that, at least as far as the Hilbert $A$-module structure goes, we can assume that $X=\Gamma_0(\HH)$ for a continuous Hilbert bundle $p:\HH\to V$, and then for each $v\in V$, since $\phi(A)$ acts on $X$ by right $A$-module maps, there is a representation $\pi_v$ of $A$ on the fibre $H_v$ such that
\[
\bigl(\phi(f)\xi\bigr)(v)=\pi_v(f)\xi(v)\midtext{for}\xi\in X.
\]

\begin{cor}\label{corres cor}
Let $p:\HH\to V$ be a continuous Hilbert bundle, let $X=\Gamma_0(\HH)$ be the associated Hilbert $A$-module, and suppose that $\phi:A\to \LL(X)$ is a nondegenerate homomorphism.
Then the $A$-correspondence $X$ is isomorphic to $X(E)$ for some topological graph $E$ if and only if $\HH$ has a continuous choice $\EE$ of orthonormal bases such that for each $v\in V$, $\EE_v=\EE\cap H_v$ diagonalizes the representation $\pi_v$ of $A$ on $H_v$ associated to the left module action $\phi:A\to \LL(X)$.
\end{cor}

\begin{proof}
First assume that $X=X(E)$ for a topological graph $E=(E^1,V,r,s)$. Then as in \thmref{hilbert thm} we can take $\EE=\gamma(E^1)$ as our continuous choice of orthonormal bases. For each $f\in A$, $v\in V$, and $e\in \EE_v$ we have
\[
\pi_v(f)\gamma(e)=\pi_v(f)\Chi_e=f(r(e))\Chi_e,
\]
so $\gamma(e)$ is an eigenvector of $\pi_v(f)$. Thus the orthonormal basis $\EE_v$ of $H_v$ diagonalizes the representation $\pi_v$.

Conversely, suppose that we have a continuous choice $\EE$ of orthonormal bases for $\HH$
such that $\EE_v$ diagonalizes $\pi_v$ for every $v\in V$.
By \thmref{hilbert thm} we may assume that $X=X(E^1,s)$ for a local homeomorphism $s:E^1\to V$, and that $\EE=\{\Chi_e:e\in E^1\}$.
Our hypothesis then says that for each $e\in E^1$, $\Chi_e$ is a common eigenvector for the operators $\{\pi_v(f):f\in A\}$,
and hence by restriction we have a $1$-dimensional representation of $A$ on the subspace $\C\Chi_e$; therefore there is a unique $r(e)\in V$ such that
\[
\pi_v(f)\Chi_e=f(r(e))\Chi_e\midtext{for}f\in A.
\]
We must show two things:
\begin{enumerate}
\item $r:E^1\to V$ is continuous, so that $E:=(E^1,V,r,s)$ is a topological graph, and

\item for $f\in A$, $\xi\in X$, and $e\in E^1$ we have
\[
\bigl(\phi(f)\xi\bigr)(e)=f(r(e))\xi(e).
\]
\end{enumerate}

For (1), let $e_0\in E^1$, and choose $\xi\in C_c(E^1)\subset X$ that is identically $1$ on a neighborhood $U$ of $e_0$. It suffices to show that for all $f\in A$ the function $e\mapsto f(r(e))$ is continuous at $e_0$. But for $e\in U$ we have
\[
f(r(e))=f(r(e))\xi(e)=\bigl(\phi(f)\xi\bigr)(e),
\]
and the function $\phi(f)\xi$ is continuous.

For (2), let $v=s(e)$, and note that
\begin{align*}
\Bigl.\bigl(\phi(f)\xi\bigr)\Bigr|_{s\inv(v)}
&=\pi_v(f)\xi|_{s\inv(v)}
\\&=\pi_v(f)\sum_{e\in s\inv(v)}\xi(e)\Chi_e
\\&=\sum_{e\in s\inv(v)}\xi(e)\pi_v(f)\Chi_e
\\&=\sum_{e\in s\inv(v)}\xi(e)f(r(e))\Chi_e,
\end{align*}
and that
\[
\Bigl.\bigl(\phi(f)\xi\bigr)\Bigr|_{s\inv(v)}
=\pi_v(f)\xi|_{s\inv(v)}
\]
so that
\[
\bigl(\phi(f)\xi\bigr)(e)=\sum_{e'\in s\inv(v)}\xi(e')f(r(e'))\Chi_e'(e)=f(r(e))\xi(e).
\qedhere
\]
\end{proof}

In order to introduce our obstruction to an $A$-correspondence coming from a topological graph, we specialize to the case of 1-dimensional fibres. As above, we separate into two steps, starting with Hilbert modules.

\begin{cor}\label{hilbert 1-dim}
Let $X$ be a Hilbert $A$-module whose associated Hilbert bundle $p:\HH\to V$ has 1-dimensional fibres. Then $X$ is isomorphic to $X(E^1,s)$ for some local homeomorphism $s:E^1\to V$ if and only if $\HH$ is trivial.
\end{cor}

\begin{proof}
First assume that $\HH$ is trivial. Then after applying a Hilbert-bundle isomorphism (equivalently, a Hilbert module isomorphism), we may assume that $\HH$ is a product bundle $H\times V$, with bundle projection equal to the coordinate projection onto the second factor. Let $B$ be an orthonormal basis of the Hilbert space $H$. Then it is routine to verify that $\EE:=B\times V$ is a continuous choice of orthonormal bases for the product bundle. Thus $X\cong X(E^1,s)$ for some local homeomorphism $s:E^1\to V$.

For the converse, it suffices to show that if $X=X(E^1,s)$ for some local homeomorphism, and if the associated Hilbert bundle $p:\HH\to V$ has 1-dimensional fibres, then $\HH$ is trivial. Define $\gamma:E^1\to \HH$ as before, and let $\EE=\gamma(E^1)$ be the corresponding continuous choice of orthonormal bases. For each $v\in V$ the intersection $\EE_v=\EE\cap H_v$ is an orthonormal basis for the fibre $H_v=\ell^2(s\inv(v))$. Since $H_v$ is 1-dimensional, the cardinality of $s\inv(v)$ is 1. Thus the local homeomorphism $s:E^1\to V$ is in fact bijective, and hence is a homeomorphism of $E^1$ onto $V$. Thus, replacing the local homeomorphism $s:E^1\to V$ by the isomorphic (in an obvious sense) one $\id:V\to V$, the Hilbert bundle $\HH$ is isomorphic to $\C\times V$, and is therefore trivial.
\end{proof}

Note that the first part of the above proof did not depend upon the 1-dimensionality hypothesis.

We can immediately parlay \corref{hilbert 1-dim} to a result involving correspondences:

\begin{cor}\label{correspondence 1-dim}
Let $X$ be an $A$-correspondence whose associated Hilbert bundle $p:\HH\to V$ has 1-dimensional fibres. Then $X$ is isomorphic to $X(E)$ for some topological graph $E$ if and only if $\HH$ is trivial.
\end{cor}

\begin{proof}
First assume that $\HH$ is trivial. As in the above proof, we can find a continuous choice $\EE$ of orthonormal bases for $\HH$. We only need to show that for each $v\in V$ the associated representation $\pi_v:A\to B(H_v)$ is diagonalized by the orthonormal basis $\EE_v=\EE\cap H_v$. But this is obvious since $H_v$ is 1-dimensional.

The converse follows immediately from \corref{hilbert 1-dim}.
\end{proof}

We will show below in \corref{picard} that
\corref{correspondence 1-dim} applies in particular to imprimitivity bimodules over $A$, and then we can make a connection with the Picard group of $A$.

We use the convention that $A$-imprimitivity bimodules are complete\footnote{unlike in some earlier literature on imprimitivity bimodules between $C^*$-algebras, for example \cite{rae:picard, bgr}}, so that such an imprimitivity bimodule $X$ is an $A$-correspondence such that the homomorphism $\phi:A \to \LL(X)$ associated to the left-module action is in fact an isomorphism onto $\KK(X)$.
By \cite[Corollary~3.33]{tfb}, every $A$-imprimitivity bimodule $X$ induces a homeomorphism of $V$, called the \emph{Rieffel homeomorphism}.

The \emph{Picard group of $A$}, denoted $\pic A$, is the group of isomorphism classes of $A$-imprimitivity bimodules with group operation balanced tensor product over $A$.
The identity element of $\pic A$ is the class of the \emph{trivial} $A$-imprimitivity bimodule, namely $A$ itself with operations
\[
f\cdot g=fg,
\quad
\<f,g\>_{A}=\bar fg,
\midtext{and}
{}_{A}\<f,g\>=f\bar g.
\]
We call the element of $\pic A$ determined by an $A$-imprimitivity bimodule $X$ the \emph{Picard invariant} of $X$.
By \cite[Proposition~3.1]{bgr} (slightly modified, to turn an antihomomorphism into a homomorphism), the group $\aut A$ of automorphisms embeds as a subgroup of $\pic A$.
More precisely, our convention differs from that of \cite{bgr}: if $\alpha\in\aut A$, we define the associated imprimitivity bimodule of $A$ to be $X=A$ with operations defined for $\xi,\eta\in X$ and $a\in A$ by
\begin{itemize}
\item $\xi\cdot a=\xi a$;
\item $\<\xi,\eta\>_A=\xi^*\eta$;
\item $a\cdot\xi=\alpha(a)\xi$;
\item ${}_A\<\xi,\eta\>=\alpha\inv(\xi\eta^*)$.
\end{itemize}

\begin{cor}\label{picard}
Let $X$ be an $A$-imprimitivity bimodule. Then $X$ is isomorphic to $X(E)$ for some topological graph $E$ if and only if the Picard invariant of $X$ is an automorphism.
\end{cor}

\begin{proof}
First we show that, because the $A$-correspondence $X$ is in fact an imprimitivity bimodule, it follows that the associated Hilbert bundle $p:\HH\to V$ has 1-dimensional fibres. For $v\in V$, recall our notation $J_v=\{f\in A:f(v)=0\}$. Let $I_v=\clspn\{{}_A\<X,X\cdot J_v\>\}$ be the ideal of $A$ that is induced by the imprimitivity bimodule $X$. Then by, e.g., \cite[Lemma~5.12]{tfb}, the quotient $H_v=X/(X\cdot J_v)$ is an $(A/I_v)-(A/J_v)$ imprimitivity bimodule. Since the ideal $J_v$ is maximal, so is $I_v$, by the properties of the Rieffel correspondence, and so $H_v$ is a $\C$-imprimitivity bimodule, and hence, again by the properties of the Rieffel correspondence, must be 1-dimensional.

Now assume that $X=X(E)$ for a topological graph $E$.
Thus by \corref{hilbert 1-dim} we can assume that the associated Hilbert bundle $p:\HH\to V$ is trivial, and consequently that as a Hilbert $A$-module $X$ is isomorphic to $A$ with operations coming from the multiplication and involution of $A$. But every $A$-imprimitivity bimodule in which the Hilbert module structure is given by $A$ with the standard operations has left module structure given by an automorphism, because in any event the left module structure is given by a homomorphism $\varphi:A\to \LL(A)=M(A)$ that is, by the properties of imprimitivity bimodules, actually an isomorphism onto $\KK(A)=A$, namely an automorphism of~$A$.

Conversely, assume that the Picard invariant of $X$ is an automorphism.
It follows from the constructions 
that the associated Hilbert bundle $\HH$ is trivial. The result now follows from \corref{correspondence 1-dim}.
\end{proof}

\begin{rem}
Let $X$ be an $A$-imprimitivity bimodule, and let $p:\HH\to V$ be the associated continuous Hilbert bundle. 
Assume that $X$ is \emph{symmetric} (borrowing terminology from \cite{AEhilbert}) in the sense that $a\cdot \xi=\xi\cdot a$ for all $a\in A$ and $\xi\in X$.
Since $\HH$ has one-dimensional fibres, $\HH$ is a complex line bundle (i.e., it is locally trivial, by \cite[Section~10, Remarque]{DixmierDouady}).
Vasselli shows in \cite[Proposition~4.3]{VasselliExtension} that, at least when $V$ is compact, the $C^*$-algebra $\OO_X$ is commutative, with spectrum the sphere bundle of $\HH$.
\end{rem}

\providecommand{\bysame}{\leavevmode\hbox to3em{\hrulefill}\thinspace}
\providecommand{\MR}{\relax\ifhmode\unskip\space\fi MR }
\providecommand{\MRhref}[2]{%
  \href{http://www.ams.org/mathscinet-getitem?mr=#1}{#2}
}
\providecommand{\href}[2]{#2}

\end{document}